\documentclass[10pt]{article}

\usepackage{enumerate}

\usepackage{amsmath}
\usepackage{amstext,amssymb,amsfonts}
\usepackage{fullpage}
\usepackage{bm}

\usepackage{todonotes}

\newcounter{mynotes}
\usepackage{nameref}
\usepackage[pagebackref,colorlinks,linkcolor=blue,filecolor = blue,
citecolor = blue, urlcolor = blue,pdfstartview=FitH]{hyperref}
\usepackage[capitalize]{cleveref}

\usepackage{amsthm}
\usepackage{thmtools}

\declaretheorem[within=section]{theorem}
\declaretheorem[sibling=theorem]{corollary}

\declaretheorem[sibling=theorem]{lemma}
\declaretheorem[sibling=theorem]{claim}

\declaretheorem[sibling=theorem]{definition}
\declaretheorem[sibling=theorem]{Lemma+Definition}
\declaretheorem[sibling=theorem]{remark}

\newenvironment{proofof}[1]{\begin{trivlist} \item {\bf Proof
#1:~~}}
  {\qed\end{trivlist}}

\crefname{proposition}{Proposition}{Propositions}
\crefname{conjecture}{Conjecture}{Conjectures}
\crefname{claim}{Claim}{Claims}
\crefname{remark}{Remark}{Remarks}

\newcounter{termcounter}
\renewcommand{\thetermcounter}{\Alph{termcounter}}
\crefname{term}{term}{terms}
\creflabelformat{term}{\textup{(#2#1#3)}}

\makeatletter
\def\term{\@ifnextchar[\term@optarg\term@noarg}
\def\term@optarg[#1]#2{%
  \textup{(#1)}%
  \def\@currentlabel{#1}%
  \def\cref@currentlabel{[][2147483647][]#1}%
  \cref@label[term]{#2}}
\def\term@noarg#1{%
  \refstepcounter{termcounter}%
  \textup{(\thetermcounter)}%
  \cref@label[term]{#1}}
\makeatother

\newcommand{\mrm}[1]{\mathrm {#1}}

\newcommand{\msf}[1]{\mathsf {#1}}

\newcommand{\ignore}[1]{}

\newcommand{\poly}{\mathrm{poly}}

\newcommand{\set}[1]{\left\{#1\right\}}

\newcommand{\abs}[1]{\lvert#1\rvert}

\newcommand{\norm}[1]{\lVert#1\rVert}

\definecolor{DSred}{rgb}{1,0,0}

\renewcommand{\leq}{\leqslant}
\renewcommand{\geq}{\geqslant}
\renewcommand{\ge}{\geqslant}
\renewcommand{\le}{\leqslant}
\renewcommand{\epsilon}{\varepsilon}
\newcommand{\eps}{\epsilon}

\newcommand{\R}{\mathbb{R}}
\newcommand{\C}{\mathbb{C}}
\newcommand{\Z}{\mathbb{Z}}
\newcommand{\N}{\mathbb{N}}
\newcommand{\F}{\mathbb{F}}
\newcommand{\T}{\mathbb{T}}
\newcommand{\U}{\mathbb{U}}

\newcommand{\cB}{\mathcal B}

\newcommand{\cF}{\mathcal F}

\newcommand{\cK}{\mathcal K}

\newcommand{\cP}{\mathcal P}

\newcommand{\Esymb}{{\bf E}}

\newcommand{\Psymb}{{\bf Pr}}

\DeclareMathOperator*{\E}{\Esymb}

\DeclareMathOperator*{\ProbOp}{\Psymb}

\renewcommand{\Pr}{\ProbOp}

\newcommand{\Ex}[1]{\E\Brac{#1}}

\newcommand{\expo}[1]{{\mathsf{e}\left(#1\right)}}

\newcommand{\Poly}{\ComplexityFont{P}}

\usepackage[margin=2cm]{geometry}   
\usepackage{url}
\usepackage{latexsym}
\usepackage{amsmath}
\usepackage{amsfonts}
\usepackage{amssymb}
\usepackage{amsthm}
\usepackage{mathrsfs}
\usepackage{enumerate}

\newtheorem{openproblem}{Open Problem}

\crefname{openproblem}{Open Problem}{Open Problems}

\renewcommand{\[}{\begin{equation}}
\renewcommand{\]}{\end{equation}}

\newcommand{\restate}[2]{\medskip
\noindent{\bf #1 (restated).}{\sl #2}}

\def\poly{\mathrm{Poly}}

\definecolor{Blue}{rgb}{0,0,1}
\definecolor{Red}{rgb}{1,0,0}

\usepackage{bbm}

\def\Z{{\mathbb{Z}}}

\def\rank{{\mathrm{rank}}}

\def\Poly{\mathrm{Poly}}

\def\tf{\widetilde{f}}

\def\rank{\text{rank}}

\def\D{\mathbb{D}}
\def\Ex{\E}
\def\bias{\mathrm{bias}}
\def\crank{\text{crank}}
\def\nrank{\text{{strong-rank}}}
\def\depth{\mathrm{depth}}
\newcommand{\ip}[1]{{\langle #1 \rangle}}
\def\cK{\mathcal{K}}

\def\iH{H^{(i)}}
\def\iG{G^{(i)}}

\title{On the Structure of Quintic Polynomials}

\author{Pooya Hatami\thanks{This material is based upon work supported by the National Science Foundation under agreement No. CCF-1412958. Any opinions, findings and conclusions or recommendations expressed in this material are those of the author and do not necessarily reflect the views of the National Science Foundation.}\\ Institute for Advanced Study, Princeton, NJ\\ \textit{pooyahat@math.ias.edu}}
\begin{document}

\date{}

\maketitle

\begin{abstract}
We study the structure of bounded degree polynomials over finite fields. Haramaty and Shpilka [STOC 2010] showed that biased degree three or four polynomials admit a strong structural property. We confirm that this is the case for degree five polynomials also. Let $\F=\F_q$ be a prime field.
\begin{itemize}
\item[1.]\label{item1} Suppose $f:\F^n\rightarrow \F$ is a degree five polynomial with $\bias(f)=\delta$. Then $f$ can be written in the form $f= \sum_{i=1}^{c} G_i H_i + Q$, where $G_i$ and $H_i$s are nonconstant polynomials satisfying $\deg(G_i)+\deg(H_i)\leq 5$ and $Q$ is a degree $\leq 4$ polynomial. Moreover, $c=c(\delta)$ does not depend on $n$ and $q$.
\item[2.] Suppose $f:\F^n\rightarrow \F$ is a degree five polynomial with $\bias(f)=\delta$. Then there exists an $\Omega_\delta(n)$ dimensional subspace $V\subseteq \F^n$ such that $f\vert_V$ is a constant. 
\end{itemize}
Cohen and Tal [Random 2015] proved that biased polynomials of degree at most four are constant on a subspace of dimension $\Omega(n)$. Item [2.] extends this to degree five polynomials. A corollary to Item [2.] is that any degree five affine disperser for dimension $k$ is also an affine extractor for dimension $O(k)$. We note that Item [2.] cannot hold for degrees six or higher.

We obtain our results for degree five polynomials as a special case of structure theorems that we prove for biased degree $d$ polynomials when $d<|\F|+4$. While the $d<|\F|+4$ assumption seems very restrictive, we note that prior to our work such structure theorems were only known for $d<|\F|$ by Green and Tao [Contrib. Discrete Math. 2009] and Bhowmick and Lovett [arXiv:1506.02047]. Using algorithmic regularity lemmas for polynomials developed by Bhattacharyya, et. al. [SODA 2015], we show that whenever such a strong structure exists, it can be found algorithmically in time polynomial in $n$. 
\end{abstract}

\newpage

\section{Introduction}
Let $\F$ be a finite field. The bias of a function $f:\F^n\rightarrow \F$ is defined as 
$$
\bias(f):=\left|\Ex_{x\in \F^n} \left[ \omega^{f(x)} \right]\right|,
$$
where $\omega= e^{2\pi i / |\F|}$, is a complex primitive root of unity of order $|\F|$. The smaller the bias of a function, the more uniformly $f$ is distributed over $\F$, thus a random function has negligible bias. This remains true, if $f$ is a random degree $d$ polynomial for a fixed degree $d>0$. Thus bias can be thought of as a notion of pseudorandomness for polynomials, and as often lack of pseudorandomness implies structure, one may ask whether every biased degree $d$ polynomial admits strong structural properties. Green and Tao~\cite{MR2592422} (in the case when $d<|\F|$) and later Kaufman and Lovett~\cite{KL08} (in the general case)  proved this heuristic to be true by showing that every biased degree $d$ polynomial is determined by a few lower degree polynomials. Formally, these results state that for a degree $d$ polynomial $f$, there is a constant $c\leq c(d,\bias(f), |\F|)$, degree $\leq d-1$ polynomials $Q_1,\ldots, Q_c$ and a function $\Gamma:\F^c\rightarrow \F$, such that
\begin{equation}\label{eq:biasstructure}
f=\Gamma(Q_1,\ldots, Q_c).
\end{equation}
Note that crucially $c$ does not depend on the dimension $n$, meaning that for large $n$, it is very unlikely for a typical polynomial to be biased. Recently, Bhowmick and Lovett~\cite{BL15largefield} proved that the dependence of the number of terms in \cref{eq:biasstructure} on $|\F|$ can be removed, in other words biased polynomials are very rare even when the field size is allowed to grow with $n$. 
These structure theorems for biased polynomials have had several important applications. For example they were used by Kaufman and Lovett~\cite{KL08} to give interesting worst case to average case reductions, and by Tao and Ziegler~\cite{MR2948765} in their proof of the inverse theorem for Gowers norms over finite fields. Such structure theorems have played an important role in determining the weight distribution and list decoding radius of Reed-Muller codes~\cite{KLP12, BL15reedmuller, BL15largefield}. They were also used by Cohen and Tal~\cite{CT15} to show that any degree $d$ affine disperser over a prime field is also an affine extractor with related parameters.

There are however two drawbacks to the structure theorems proved in \cite{MR2592422, KL08}. Firstly, the constant $c=c(\delta,d,|\F|)$ has very bad dependence on $\delta$ which is due to the use of regularity lemmas for polynomials. Secondly, there is no restrictions on the function $\Gamma$ obtained in  \cref{eq:biasstructure}, in particular there is nothing stopping it from being of degree $c$. In the special case of quadratic polynomials better bounds and structural properties follow from the following well-known theorem.
\begin{theorem}[Structure of quadratic polynomials~\cite{lidl}]
For every quadratic polynomial $f:\F^n\rightarrow \F$ over a prime field $\F$, there exists an invertible linear map $T$, a linear polynomial $\ell$, and field elements $\alpha_1,\ldots, \alpha_n$ such that
\begin{itemize}
\item If $|\F|=2$, then $(f \circ T)(x)= \sum_{i=1}^{\lfloor n/2\rfloor} \alpha_i x_{2i-1}x_{2i} + \ell(x)$.
\item If $|\F|$ is odd, then $(f \circ T)(x)= \sum_{i=1}^n \alpha_i x_i^2+ \ell(x).$
\end{itemize}
\end{theorem}
It easily follows that every quadratic polynomial $f$, can be written in the form $\sum_{i=1}^{2\log(1/\bias(f))} \ell_i \ell'_i+ \ell''$ where $\ell_i, \ell'_i$s and $\ell''$ are linear polynomials. This is a very strong structural property, moreover the dependence of the number of the terms on $\bias(f)$ is optimal. Haramaty and Shpilka~\cite{MR2743281} studied the structure of biased cubic and quartic polynomials and proved the following two theorems.
\begin{theorem}[Biased cubic polynomials~\cite{MR2743281}]\label{HS:cubic}
Let $f:\F^n \rightarrow \F$ be a cubic polynomial such that $\bias(f)=\delta>0$. Then there exist $c_1=O\left(\log(1/\delta)\right)$, $c_2=O\left(\log^4(1/\delta)\right)$, quadratic polynomials $Q_1,...,Q_{c_1}:\F^n\rightarrow \F$, linear functions $\ell_1,...,\ell_{c_1}, \ell'_1,...,\ell'_{c_2}:\F^n\rightarrow \F$ and a cubic polynomial $\Gamma:\F^{c_2}\rightarrow \F$ such that
$$
f= \sum_{i=1}^{c_1} \ell_i Q_i + \Gamma\left(\ell'_1,\ldots, \ell'_{c_2}\right).
$$ 
\end{theorem}
\begin{theorem}[Biased quartic polynomials~\cite{MR2743281}]\label{HS:quartic}
Let $f:\F^n \rightarrow \F$ be a cubic polynomial such that $\bias(f)=\delta$. There exist $c=\poly(|\F|/\delta)$ and polynomials $\{\ell_i, Q_i, Q'_i, G_i\}_{i\in [c]}$, where the $\ell_i$s are linear, $Q_i, Q'_i$s are quadratic, and $G_i$'s are cubic polynomials, such that
$$f=\sum_{i=1}^c \ell_i G_i + \sum_{i=1}^c Q_iQ'_i.$$
\end{theorem}
In the high characteristic regime when $d=\deg(f)<|\F|$, Green and Tao~\cite{MR2592422} showed that such a strong structure theorem holds, with a dependence that is really large in terms of bias. More precisely, if $d<|\F|$, then every degree $d$ polynomial $f$, with $\bias(f)\geq \delta$ can be written in the form 
$
f= \sum_{i=1}^{c(\delta, \F,d)} G_i H_i + Q,
$
where $G_i$ and $H_i$s are nonconstant polynomials satisfying $\deg(G_i)+\deg(H_i)\leq d$, and $Q$ is a degree $\leq d-1$ polynomial. Recently, Bhowmick and Lovett~\cite{BL15largefield} have proved that one can remove the dependence of $c$ on $|\F|$.

\section*{Our results}
Suppose that $\F=\F_q$ is a prime field. When the characteristic of $\F$ can be small, it was not known whether a degree five biased polynomial admits a strong structure in the sense of  \cref{HS:cubic,HS:quartic}. Moreover, the techniques from \cite{MR2743281} seem to break down.  

\paragraph{Quintic polynomials.} We combine ideas from \cite{MR2743281} with arguments from polynomial regularity and prove such a structure theorem for quintic polynomials.
\begin{theorem}[Biased quintic polynomials I]\label{main}
Suppose $f:\F^n \rightarrow \F$ is a degree five polynomial with $\bias(f)=\delta$. 
There exist $c_{\ref{main}}\leq c(\delta)$, nonconstant polynomials $G_1,...,G_c, H_1,...,H_c$ and a polynomial $Q$ such that the following holds.
\begin{itemize}
\item $f= \sum_{i=1}^c G_i H_i+Q$.
\item For every $i\in [c]$, $\deg(G_i)+\deg(H_i)\leq 5$.
\item $\deg(Q)\leq 4$.
\end{itemize}
\end{theorem}
Note that $c_{\ref{main}}$ only depends on $\delta$, and has no dependence on $n$ or $|\F|$.
%
We also prove that every biased quintic polynomial is constant on an affine subspace of dimension $\Omega(n)$.
\begin{theorem}[Biased quintic polynomials II]\label{main2}
Suppose $f:\F^n \rightarrow \F$ is a degree five polynomial with $\bias(f)=\delta$. There exists an affine subspace $V$ of dimension $\Omega(n)$ such that $f\vert_V$ is constant, where the constant hidden in $\Omega$ depends only on $\delta$.
\end{theorem}
\cref{main2} was previously only known for degrees $\leq 4$. The case of quadratics when $\F=\F_2$ is Dickson's theorem~\cite{MR0104735}, and the case of general $\F$ and $d\leq 4$ was proved recently by Cohen and Tal~\cite{CT15} building on \cref{HS:cubic,HS:quartic}. We also remark that the degree five is the largest degree that such a bound can hold. To see this, assume for example that $d=6$ and $\F=\F_2$, and construct a degree $6$ polynomial $f=G(x_1,...,x_n)\cdot H(x_1, \ldots, x_n)$ by picking two random cubic polynomials $G$ and $H$. One observes that $f$ has bias very close to $0$, however, $f$ will not vanish over any subspace of dimension $\Omega(n^{1/2})$.
\cref{main2} has the following immediate corollary.
\begin{corollary}
Suppose $f:\F^n \rightarrow \F$ is a degree five affine disperser for dimension $k$.  Then $f$ is also an affine extractor of dimension $O(k)$.
\end{corollary} 
We refer to \cite{CT15} where affine dispersers and extractors and the relations between them are discussed.

\paragraph{Degree $d$ polynomials, with $d<|\F|+4$.} We in fact prove a strong structure theorem for biased degree $d$ polynomials when $d<|\F|+4$, from which \cref{main} follows immediately.
\begin{theorem}[Biased degree $d$ polynomials I (when $d<|\F||\F|+4$)]\label{generald}
Suppose $d>0$ and $\F=\F_{q}$ with $d< q+4$. Let $f:\F^n \rightarrow \F$ be a degree $d$ polynomial with $\bias(f)=\delta$. There exists $c_{\ref{generald}}\leq c(\delta, d)$, nonconstant polynomials $G_1,...,G_c, H_1,...,H_c$ and a polynomial $Q$ such that the following hold.
\begin{itemize}
\item $f= \sum_{i=1}^c G_i H_i+Q$.
\item For every $i\in [c]$, $\deg(G_i)+\deg(H_i)\leq d$.
\item $\deg(Q)\leq d-1$.
\end{itemize}
\end{theorem}
We also prove a general version of  \cref{main2} for $d<|\F|+4$.
\begin{theorem}[Biased degree $d$ polynomials II (when $d<|\F|+4)$]\label{generald2} 
Suppose $d>0$ and $\F=\F_q$ with $d<q+4$. Let $f:\F^n \rightarrow \F$ be a degree $d$ polynomial with $\bias(f)=\delta$. There exists an affine subspace $V$ of dimension $\Omega_{d,\delta}(n^{1/\lfloor \frac{d-2}{2}\rfloor})$ such that $f\vert_V$ is a constant.
\end{theorem}
Cohen and Tal~\cite{CT15} recently showed that any degree $d$ biased polynomial is constant on an $\Omega_{\delta}(n^{1/(d-1)})$ dimensional affine subspace. \cref{generald2} improves on this by a quadratic factor, when $d<|\F|+4$.

Our results for quintic polynomials follow immediately.
\begin{proofof}{of \cref{main,main2}}
\cref{main,main2} follow curiously as special cases of \cref{generald} and \cref{generald2} as $|\F|\geq 2$ and $5<2+4$. 
\end{proofof}
\paragraph{Algorithmic aspects.} Using a result of Bhattacharyya, et. al.~\cite{BHT15} who gave an algorithm for finding prescribed decompositions of polynomials, we show that whenever such a strong structure exists, it can be found algorithmically in time polynomial in $n$. Combined with \cref{generald}, we obtain the following algorithmic structure theorem.
\begin{theorem}\label{algorithmic}
Suppose $\delta>0$, $d>0$ are given, and let $\F=\F_q$ be a prime field satisfying $d<q+4$. There is a deterministic algorithm that runs in time $O(n^{O(d)})$ and given as input a degree $d$ polynomial $f:\F^n\to \F$ satisfying $\bias(f)=\delta$, outputs a number $c\leq c(\delta, |\F|,d)$, a collection of degree $\leq d-1$ polynomials $G_1,...,G_c, H_1,...,H_c:\F^n\to \F$ and a polynomial $Q:\F^n\to \F$, such that
\begin{itemize}
\item $f= \sum_{i=1}^c G_iH_i +Q$.
\item For every $i\in [c]$, $\deg(G_i)+\deg(H_i)\leq d$.
\item $\deg(Q)\leq d-1$.
\end{itemize}
\end{theorem}

\section*{Organization}
In \cref{sec:higher} we present the basic tools from higher-order Fourier analysis. In \cref{generic} we discuss useful properties of a pseudorandom collection of polynomials. \cref{generald} is proved in \cref{section:generald}, and \cref{generald2} is proved in \cref{section:generald2}. We discuss the algorithmic aspects in \cref{section:algorithmic}. We end with a discussion of future directions in \cref{section:conclusions}.

\section*{Notation}
Let $\D=\{z \in \C: |z| \le 1\}$ be the unit disk in the complex plane. Let $\T= \R/\Z$.  Suppose that $\F=\F_q$ is a finite prime field, let $e_\F:\F\rightarrow \D$ denote the function $e_\F(x):= e^{\frac{2\pi i x}{|\F|}}$, and let $e:\T\to\D$ denote the function $e(x):=e^{2\pi i x}$. For functions $f,g:\F^n\rightarrow \C$, define $$\langle f, g\rangle:= \frac{1}{|\F|^n} \sum_{x\in \F^n} f(x) \overline{g(x)}.$$ 
For an integer $a$, denote by $[a]:=\{1,\ldots,a\}$. 

\section{Preliminary results from higher-order Fourier analysis}\label{sec:higher}
Throughout this section, assume that $\F=\F_q$ for a fixed prime $q$. Extensions to large finite fields will be discussed later in \cref{generalF}.

\subsection{Nonclassical Polynomials}

Let $d\geq 0$ be an integer. It is well-known that for functions $P:\R^n\rightarrow \R$, a polynomial of degree $\le d$ can be defined in one of two equivalent ways: We say that $P$ is a polynomial of degree $\le d$ if it can be written as $$
P(x_1,...,x_n)=\sum_{\substack{i_1,...,i_n \ge 0\\ i_1+\cdots+i_n \le d}} c_{i_1,...,i_n}x_1^{i_1}\cdots x_n^{i_n},
$$
with coefficients $c_{i_1,...,i_n}\in \R$. This can be thought of as a global definition for polynomials over the reals. Equivalently, the local way of defining $P$ to be a polynomial of degree $\le d$ is to say that it is $d+1$ times differentiable and its $(d+1)$-th derivative vanishes everywhere.

In finite characteristic, i.e. when  $P:\F^n\rightarrow G$ for a prime field $\F$ and an abelian group $G$, the local definition of a polynomial uses the notion of additive directional derivatives.

\begin{definition}[Polynomials over finite fields (local definition)]
For an integer $d \geq 0$, a function $P: \F^n \to G$ is said to be a
{\em polynomial of degree $\leq d$} if for all $y_1,
\dots, y_{d+1}, x \in \F^n$, it holds that
\begin{equation*}
(D_{y_1}\cdots D_{y_{d+1}} P)(x) = 0,
\end{equation*}
where $D_y P(x)=P(x+y)-P(x)$ is the additive derivative of $P$ with direction $y$ evaluated at $x$. The {\em degree} of $P$ is the smallest $d$ for which the above holds.
\end{definition}

It follows simply from the definition that for any direction $y\in \F^n$, $deg(D_yP)<deg(P)$. In the ``classical'' case of polynomials $P:\F^n\rightarrow \F$, it is a well-known fact that the global and local definitions coincide. However, the situation is different when $G$ is allowed to be other groups.\ignore{One way to suspect this is the fact that we are not able to divide by $d!$ in some cases, in order to be able to make use of Taylor expansions.} For example when the range of $P$ is $\R/\Z$, it turns out that the global definition must be refined to the ``nonclassical polynomials''. This phenomenon was noted by Tao and Ziegler~\cite{MR2948765} in the study of Gowers norms.

Nonclassical polynomials arise when studying functions $P:\F^n \to \T$ and their exponents $f=e(P):\F^n \to \C$.

\begin{definition}[Nonclassical Polynomials]\label{poly}
For an integer $d \geq 0$, a function $P: \F^n \to \T$ is said to be a
{\em nonclassical polynomial of degree $\leq d$} (or simply a
{\em polynomial of degree $\leq d$}) if for all $y_1,
\dots, y_{d+1}, x \in \F^n$, it holds that
\begin{equation}\label{eqn:poly}
(D_{y_1}\cdots D_{y_{d+1}} P)(x) = 0.
\end{equation}
The {\em degree} of $P$ is the smallest $d$ for which the above holds.
A function $P : \F^n \to \T$ is said to be a {\em classical polynomial of degree
$\leq d$} if it is a nonclassical polynomial of degree $\leq d$
whose image is contained in $\frac{1}{q} \Z/\Z$.

Denote by $\poly(\F^n\rightarrow \T)$, $\poly_d(\F^n\rightarrow \T)$ and $\poly_{\leq d}(\F^n\rightarrow \T)$, the set of all nonclassical polynomials over $\F^n$, all nonclassical polynomials of degree $d$ and all nonclassical polynomials of degree $\leq d$ respectively.
\end{definition}
The following lemma of Tao and Ziegler~\cite{MR2948765} shows that a classical
polynomial $P$ of degree $d$ must always be of the form $x \mapsto \frac{|Q(x)|}{q}$, where $Q : \F^n \to \F$ is a polynomial (in the usual sense) of degree $d$, and $|\cdot|$ is the standard map from $\F$ to
$\set{0,1,\dots,q-1}$. This lemma also characterizes the structure of nonclassical
polynomials.

\begin{lemma}[Lemma 1.7 in \cite{MR2948765}]\label{struct}
A function $P: \F^n \to \T$ is a polynomial of degree $\leq d$ if and
only if $P$ can be represented as
$$P(x_1,\dots,x_n) = \alpha + \sum_{\substack{0\leq d_1,\dots,d_n< q; k \geq 0:
  \\ {0 < \sum_i d_i \leq d - k(q-1)}}} \frac{ c_{d_1,\dots, d_n,
  k} |x_1|^{d_1}\cdots |x_n|^{d_n}}{q^{k+1}} \mod 1,
$$
for a unique choice of $c_{d_1,\dots,d_n,k} \in \set{0,1,\dots,q-1}$
and $\alpha \in \T$.  The element $\alpha$ is called the {\em
  shift} of $P$, and the largest integer $k$ such that there
exist $d_1,\dots,d_n$ for which $c_{d_1,\dots,d_n,k} \neq 0$ is called
the {\em depth} of $P$. A depth-$k$ polynomial $P$ takes values in an affine shift of the subgroup $\U_{k+1}:= \frac{1}{q^{k+1}} \Z/\Z$. Classical polynomials correspond to
polynomials with $0$ shift and $0$ depth.
\end{lemma}
In many cases, for the sake of brevity, we will omit writing ``$\mathrm{mod}\; 1$'' in the description of the defined nonclassical polynomials. 
%
For convenience of exposition, henceforth we will assume
that the shifts of all polynomials are zero. This can be done
without affecting any of the results presented in this text. Under this assumption, all
polynomials of depth $k$ take values in $\U_{k+1}$.

\subsection{Gowers norms}
Gowers norms, which were introduced by Gowers~\cite{Gow01}, play an important role in additive combinatorics, more specifically in the study of polynomials of bounded degree. Gowers norms are defined for functions $F:G \to \C$, where $G$ is any finite Abelian group. In this paper we will restrict our attention to the case of $G=\F^n$.
\ignore{\begin{definition}[Multiplicative Derivative]
The (multiplicative) derivative of $f$ in direction $y \in \F^n$ is given by $\Delta_y f(x)=f(x+y) \overline{f(x)}$.
\end{definition}
Note that if $f(x)=(-1)^{f(x)}$ then $\Delta_y F = (-1)^{D_y f}$.}
The \emph{Gowers norm} of order $d$ for $f$ is defined as the expected $d$-th multiplicative derivative of $f$ in $d$ random directions at a random point.
\begin{definition}[Gowers norm]\label{gowers}
Let $\F=\F_q$ be a finite field, $d >0$.
Given a function $f: \F^n \to \C$, the {\em
  Gowers norm of order $d$} for $f$ is given by
\begin{equation*}
\|f\|_{U^d} := \left|\E_{y_1,\dots,y_d,x \in \F^n} \left[(\Delta_{y_1}\Delta_{y_2} \cdots \Delta_{y_d}f)(x)\right]\right|^{1/2^d} 
= \left|\Ex_{y_1,\dots,y_d,x \in \F^n} \left[\prod_{S \subseteq [d]} \mathcal{C}^{d-|S|} f(x + \sum_{i \in S} y_i) \right]\right|^{1/2^d},
\end{equation*}
where $\mathcal{C}$ is the conjugation operator $\mathcal{C}(z)=\overline{z}$ and $\Delta_y f(x)= f(x+y) \overline{f(x)}$ is the multiplicative derivative of $f$ at direction $y$.
\end{definition}
Note that as $\|f\|_{U^1}= |\Ex\left[ f \right]|$ the Gowers norm of order $1$ is only a semi-norm. However for $d>1$, it turns out that
$\|\cdot\|_{U^d}$ is indeed a norm~\cite{Gow01}. Direct and inverse theorems for Gowers norms relate the Gowers norm to correlation with bounded degree polynomials. 

\begin{theorem}[Direct theorem for Gowers Norm]\label{thm:directgowers}
Let $f:\F^n\rightarrow \C$ be a function and $d \ge 1$ an integer. Then for every degree-$d$ nonclassical polynomial $P:\F^n \to \T$,
$$
|\ip{f, e(P)}| \leq \|f\|_{U^{d+1}}.
$$
\end{theorem}
%
%
\ignore{It was proved independently by Lovett, Meshulam and Samorodnitsky~\cite{MR2862496} and Green and Tao~\cite{MR2592422} that the inverse direction (previously known as ``Inverse conjecture for Gowers norms'') is false, when we restrict ourselves to classical polynomials. However, Tao and Ziegler~\cite{MR2948765} proved that it is true if we include nonclassical polynomials (\cref{poly}).} 
\begin{theorem}[Inverse theorem for Gowers Norms~\cite{MR2948765}]\label{inverse}
Fix $d \ge 1$ an integer and $\eps > 0$.
There exists an $\delta = \delta(\eps,d, |\F|)$ such that the following
holds. For every function $f: \F^n \to \D$ with $\|f\|_{U^{d+1}} \geq \eps$, there exists a polynomial $P\in \poly_{\leq d}(\F^n\rightarrow \T)$ that is $\delta$-correlated with $f$, that is
$$
|\ip{f, e(P)}| \ge \delta.
$$
\end{theorem}
\ignore{%
It is easy to see that for every degree $d$ nonclassical polynomial $P$, $\|e(P)\|_{U^{d+1}} =
1$. \cref{thm:directgowers} and \cref{inverse} provide a robust version of this statement.}

\subsection{Rank, Regularity, and Other Notions of Uniformity}
The rank of a polynomial is a notion of its complexity according to lower degree polynomials.
\begin{definition}[Rank of a polynomial]\label{def:rankpoly}
Given a polynomial $P : \F^n \to \T$ and an integer $d \ge 1$, the {\em $d$-rank} of
$P$, denoted $\msf{rank}_d(P)$, is defined to be the smallest integer
$r$ such that there exist polynomials $Q_1,\dots,Q_r:\F^n \to \T$ of
degree $\leq d-1$ and a function $\Gamma: \T^r \to \T$ satisfying
$P(x) = \Gamma(Q_1(x),\dots, Q_r(x))$. If $d=1$, then
$1$-rank is defined to be $\infty$ if $P$ is non-constant and $0$
otherwise.

The {\em rank} of a polynomial $P: \F^n \to \T$ is its $\deg(P)$-rank. We say that $P$ is $r$-regular if $\msf{rank}(P) \ge r$.
\end{definition}
Note that for an integer $\lambda \in [1, q-1]$,  $\mrm{rank}(P) =
\mrm{rank}(\lambda P)$. In this article we are interested in obtaining a structure theorem for biased classical polynomials that does not involve nonclassical polynomials. Motivated by this, we define two other notions of rank. 
\begin{definition}[Classical rank of a polynomial]\label{def:classicalrank}
Given a (classical) polynomial $P:\F^n \rightarrow \F$ and an integer $d\geq 1$, the classical $d$-rank of $P$, denoted by $\crank_d(P)$, is defined similarly to \cref{def:rankpoly} with the extra restriction that $Q_1,...,Q_r:\F^n\rightarrow \F$ are classical polynomials. 

The {\em classical rank} of a polynomial $P:\F^n\to \F$ is its classical $\deg(P)$-rank. We say that $P$ is classical $r$-regular if $\msf{crank}(P) \ge r$.
\end{definition}

\begin{remark}
For a nonconstant affine-linear polynomial $P(x)$, $\rank(P)=\crank(P)=\infty$ and for a constant function $Q(x)$, $\rank(Q)=0$.  
\end{remark}

\begin{remark}\label{rem:rankcrank}
It is important to note that \cref{def:rankpoly} and \cref{def:classicalrank} are not equivalent. To see this, note that, as proved in \cite{MR2948765} and \cite{MR2862496}, the degree $4$ symmetric polynomial $S_4:=\sum_{i<j<k<\ell} x_ix_jx_kx_\ell$ has negligible correlation with any degree $\leq 3$ classical polynomial. A simple Fourier analytic argument implies that $\crank(S_4)=\omega(1)$, i.e. $\lim_{n\rightarrow \infty} \crank(S_4(x_1,...,x_n))=\infty$. However, $\|e(S_4)\|_{U^4}\approx \frac{1}{8}$, and by a theorem of Tao and Ziegler~\cite{MR2948765} stating that functions with large Gowers norm must have large rank, we have that $\rank(S_4) \leq r(\F)$ for some constant $r$. 
\end{remark}

In the above definitions of rank of a polynomial, we have allowed the function $\Gamma$ to be arbitrary. It is interesting to ask whether a polynomial is structured in a stronger sense.

\begin{definition}[Strong rank of a polynomial]\label{def:strongrank}
Given a (classical) polynomial $P:\F^n \rightarrow \F$ of degree $d$. The {\em strong rank} of $P$, denoted by $\nrank_d(P)$, is the smallest $r\geq 0$, such that there exist nonconstant polynomials $G_1,...,G_r, H_1,...,H_r:\F^n\rightarrow \F_n$ and a polynomial $Q$ such that
\begin{itemize}
\item $P(x)= \sum_{i=1}^r G_i H_i+Q$.
\item For all $i\in [r]$, we have that $\deg(G_i)+\deg(H_i)\leq d$.
\item $\deg(Q)\leq d-1$.
\end{itemize}
The {\em strong-rank} of a polynomial $P:\F^n\to \F$ is equal to $\nrank_{\deg(P)}(P)$. \end{definition}
The above notion of rank is a stronger notion, and in particular the following holds for any polynomial $P$,
\begin{equation}\label{eq:ranks}
\rank(P)\leq \crank(P)\leq \nrank(P).
\end{equation}
Due to the lack of multiplicative structure in $\frac{1}{p^k}\Z/\Z$ for $k>1$, it is not clear how to define a similar structural notion to strong rank for nonclassical polynomials.
Next, we will formalize the
notion of a generic collection of polynomials. Intuitively, it should
mean that there are no unexpected algebraic dependencies among the
polynomials. First, we need to set up some notation.

\begin{definition}[Factors] If $X$ is a finite set then by a \emph{factor} $\cB$ we simply mean a
partition of $X$ into finitely many pieces called \emph{atoms}.
\end{definition}

A finite collection of functions $\phi_1,\ldots,\phi_C$ from $X$ to some other space $Y$ naturally define a factor $\cB=\cB_{\phi_1,\ldots,\phi_C}$ whose atoms are sets of the form $\{x: (\phi_1(x),\ldots,\phi_C(x))= (y_1,\ldots,y_C) \}$ for some $(y_1,\ldots,y_C) \in Y^C$. By an abuse of notation
we also use $\cB$ to denote the map $x \mapsto (\phi_1(x),\ldots,\phi_C(x))$, thus also identifying the atom containing $x$ with
$(\phi_1(x),\ldots,\phi_C(x))$.

\begin{definition}[Polynomial factors]\label{factor}
If $P_1, \dots, P_C:\F^n \to \T$ is a sequence of polynomials, then the factor $\cB_{P_1,\ldots,P_C}$ is called a {\em polynomial factor}.

The {\em complexity} of $\cB$, denoted $|\cB|:=C$, is the number of defining polynomials. The {\em degree} of $\cB$ is the maximum degree among its defining polynomials $P_1,\ldots,P_C$. If $P_1,\ldots,P_C$ are of depths $k_1,\ldots,k_C$, respectively, then the number of atoms of $\cB$ is at most $\prod_{i=1}^C q^{k_i+1}$ which we denote by $\|\cB\|$.
\end{definition}

The notions of rank discussed above can now be extended to quantify the structural complexity of a collection of polynomials.

\begin{definition}[Rank, classical rank, and strong rank of a collection of polynomials]\label{def:rankfactor}
A polynomial factor $\cB$ defined by polynomials $P_1,\ldots,P_C:\F^n \rightarrow \T$ with respective depths $k_1,\ldots,k_C$ is said to have rank $r$ if $r$ is the least integer for which there exists $(\lambda_1, \ldots, \lambda_C)\in \Z^C$, with $(\lambda_1 \mod q^{k_1+1}, \ldots, \lambda_C \mod q^{k_C+1})\neq 0^C$, such that $\rank_d(\sum_{i=1}^C \lambda_iP_i) \leq r$, where $d=\max_i \deg(\lambda_i P_i)$.

Given a collection of polynomials $\cP$ and a function $r:\N\rightarrow \N$, we say that $\cP$ is $r$-regular if $\cP$ is of rank larger than $r(|\cP|)$. We extend \cref{def:classicalrank} and \cref{def:strongrank} to (classical) polynomial factors in a similar manner.
\end{definition}

Notice that by the definition of rank, for a degree-$d$ polynomial $P$ of depth $k$ we have
$$
\rank(\{P\})= \min\left\{ \rank_d(P),\rank_{d-(q-1)}(qP),\ldots, \rank_{d-k(q-1)}(q^kP)\right\},
$$
where $\{P\}$ is a polynomial factor consisting of one polynomial $P$.

In \cref{generic} we will see that regular collections of polynomials indeed do behave like a generic collection of polynomials in several manners. 
\ignore{
\subsection{Bias of a polynomial}\label{analyticuniformity}}
Green and Tao~\cite{MR2592422} and Kaufman and Lovett~\cite{KL08} proved the following relation between bias and rank of a polynomial.
\begin{theorem}[$d<|\F|$ \cite{MR2592422}, arbitrary $\F$ \cite{KL08}]\label{rankreg}
For any $\eps > 0$ and integer $d \ge 1$, there exists
$r = r(d,\eps, |\F|)$ such that the following is true.
If $P: \F^n \to \T$ is a degree-$d$ polynomial $\bias(P)\geq \eps$ then $\crank(P)\leq r$.

More importantly, there are $y_1,\ldots, y_r\in \F^n$, and a function $\Gamma:\F^r \rightarrow \F$, such that
$$
P= \Gamma(D_{y_1}P, \ldots, D_{y_r}P).
$$
\end{theorem}
Kaufman and Lovett originally proved \cref{rankreg} for classical polynomials and classical rank. However, their proof
extends to nonclassical polynomials without modification. Note that $r(d,\eps,|\F|)$ does not depend on the dimension $n$.
Motivated by \Cref{rankreg} we define unbiasedness for polynomial factors.
\begin{definition}[Unbiased collection of polynomials]\label{dfn:uniformfactor}
Let $\eps:\N\rightarrow \R^+$ be a decreasing function. A polynomial factor $\cB$ defined by polynomials $P_1,\ldots, P_C:\F^n \rightarrow \T$ with respective depths $k_1,\ldots, k_C$ is said to be $\eps$-unbiased if for every collection $(\lambda_1,\ldots, \lambda_C)\in \Z^C$, with $(\lambda_1 \mod p^{k_1+1}, \ldots, \lambda_C \mod p^{k^C+1})\neq 0^C$ it holds that
$$
\left|\Ex_x \left[ \expo{\sum_i \lambda_i P_i(x)}\right] \right|<\eps(|\cB|).
$$
\end{definition}
\ignore{Using Gowers norms, one can define the following analytic notion of uniformity for polynomials which is stronger than unbiasedness.
\begin{definition}[Uniformity]\label{def:uniformpoly}
Let $\eps>0$ be a real. A degree-$d$ polynomial $P:\F^n\rightarrow \T$ is said to be $\eps$-uniform if
$$
\norm{\expo{P}}_{U^d}<\eps.
$$
\end{definition}
Tao and Ziegler used \cref{rankreg} to show that high rank polynomials have small Gowers norm.
\begin{theorem}[Theorem 1.20 of \cite{MR2948765}]\label{thm:taoziegler}\label{arank}
For any $\eps>0$ and integer $d \ge 1$, there exists an integer $r(d,|\F|,\eps)$ such that the following is true. For any nonclassical polynomial $P:\F^n\rightarrow \T$ of degree $\leq d$, if $\norm{\expo{P}}_{U^d}\geq \eps$, then $\rank_d(P)\leq r$.
\end{theorem}
This immediately implies that a regular polynomial is also uniform.
\begin{corollary}\label{cor:rankvsuniformity}
Let $\eps, d,$ and $r(d,|\F|,\eps)$ be as in \Cref{thm:taoziegler}. Every $r$-regular polynomial $P$ of degree $d$ is also $\eps$-uniform.
\end{corollary}
}

\subsection{Regularization of Polynomials}\label{intro:regularization}
Due to the generic properties of regular factors, it is often useful to {\em refine} a collection of polynomial to a regular collection~\cite{MR2948765}. We will first formally define what we mean by refining a collection of polynomials.

\begin{definition}[Refinement] \label{refine}
A collection $\cP'$ of polynomials is called a {\em refinement} of $\cP=\{P_1,...,P_m\}$, and
denoted $\cB' \succeq \cB$, if the
induced partition by $\cB'$ is a combinatorial refinement of the partition
induced by $\cB$. In other words, if for every $x,y\in \F^n$,
$\cB'(x)=\cB'(y)$ implies $\cB(x)=\cB(y)$.
\end{definition}

One needs to be careful about distinguishing between two types of refinements.

\begin{definition}[Semantic and syntactic refinements] \label{semsynrefine}
$\cB'$ is called a {\em syntactic refinement} of $\cB$, and
denoted $\cB' \succeq_{syn} \cB$, if the sequence of polynomials
defining $\cB'$ extends that of $\cB$. It is called a {\em
  semantic refinement}, and denoted $\cB' \succeq_{sem} \cB$ if the
induced partition is a combinatorial refinement of the partition
induced by $\cB$. In other words, if for every $x,y\in \F^n$,
$\cB'(x)=\cB'(y)$ implies $\cB(x)=\cB(y)$.
\end{definition}
Clearly, being a syntactic refinement is stronger than being a semantic refinement.
Green and Tao~\cite{MR2592422}, showed that given any nondecreasing function $r:\N\to\N$, any classical polynomial factor can be refined to an $r$ classical-rank factor.
The basic idea is simple; if some polynomial has low rank, decompose it to a few lower degree polynomials, and repeat. Formally, it follows
by transfinite induction on the number of polynomials of each degree which defines the polynomial factor.
The bounds on the number of polynomials obtained in the regularization process have Ackermann-type dependence on the degree $d$, even when the regularity parameter $r(\cdot)$ is a ``reasonable" function. As such,
it gives nontrivial results only for constant degrees. The extension of this regularity lemma to nonclassical polynomials is more involved, and was proved by Tao and Ziegler~\cite{MR2948765} as part of their proof of the inverse Gowers theorem (\cref{inverse}).
\begin{theorem}[Regularity lemma for nonclassical polynomials~\cite{MR2948765}]\label{factorreg}
Let $r: \N \to \N$ be a non-decreasing function and $d \ge 1$
be an integer. Then, there is a  function
$C_{\F,r,d}: \N \to \N$  such
that the following holds. Suppose $\cB$ is a factor defined by
polynomials $P_1,\dots, P_C : \F^n \to \T$  of degree at most $d$.
Then, there is  an $r$-regular factor $\cB'$ consisting of  polynomials
$Q_1, \dots, Q_{C'}: \F^n \to \T$ of degree $\leq d$ such that $\cB'
\succeq_{sem} \cB$ and  $C' \leq  C_{\ref{factorreg}}^{(\F,r,d)}(C)$.

Moreover, if $\cB$ is itself a syntactic refinement of some $\cB_0$ that has rank $>r(C')$, then additionally $\cB'$ will be a syntactic refinement of $\cB_0$.
\end{theorem}

\subsection{\ignore{Non-prime or }Growing field size}\label{generalF}
Note that in all the results discussed in this section, we have assumed that the field $\F=\F_q$ is a prime field for a fixed prime $q$, and thus the parameters that depended on $|\F|$ could be thought of as constants.

Recently, Bhowmick and Lovett~\cite{BL15largefield} proved that the dependence in the field-size can be dropped in several of the tools from higher-order Fourier analysis, allowing to extend many of the discussed results to the scenario when $\F$ can be a field with size growing with $n$. 
 
The main theorem towards obtaining such improvements is the following improvement of \cref{rankreg}.
\begin{theorem}[\cite{BL15largefield}]\label{thm:BL15}
Let $|\F|=\F_{q}$, and $d,s\in \N$. Let $P:\F^n\to \F$ be a degree $d$ polynomial. Suppose that $\bias(P)\geq |\F|^{-s}$. Then, there exist $c=c(d,s)$, polynomials $Q_1,\ldots, Q_c$, and $\Gamma:\F^c\to \F$, such that $P=\Gamma(Q_1,\ldots, Q_c)$. 
\end{theorem}
Note that $c=c(d,s)$ does not depend on $|\F|$, and remains a constant even when the field size grows with $n$. \ignore{ In \cite{BL15largefield}, the above theorem is further extended to non-prime fields $\F=\F_{q^m}$ of size possibly growing with $n$, under the assumption that $q>d$.} We list below the immediate implications to the results discussed in this section.
\begin{enumerate}
\item The dependence of $r_{\ref{rankreg}}(d,\eps, |\F|)$ on $|\F|$ in \cref{rankreg}\ignore{and \cref{thm:taoziegler}} can be removed.\ignore{ when $\F=\F_q$ is a prime field, or $\F=\F_{q^m}$ for a prime $q>d$.}
\ignore{\item The dependence of $r(d,|\F|, \epsilon)$ on $|\F|$ in 
\cref{cor:rankvsuniformity}  can be removed.  }
\item The dependence of $C_{\ref{factorreg}}^{\F,r,d}$ in \cref{factorreg} can be removed. \ignore{when $\F=\F_q$ is a prime field, or $\F=\F_{q^m}$ for a prime $q>d$.}
\item The dependence of $r_{d,|\F|}$ in $|\F|$ in \cref{lem:degreepreserve}  can be removed.\ignore{ when $\F=\F_q$ is a prime field, or $\F=\F_{q^m}$ for a prime $q>d$.}
\end{enumerate}

\section{Properties of $\rank$, $\crank$, and $\nrank$}\label{generic}
A high-rank polynomial of degree $d$ is, intuitively, a ``generic''
degree $d$ polynomial; there are no unexpected ways to decompose it
into lower degree polynomials. In this section we make precise this intuition.

Using a standard observation that relates the bias of a function to
its distribution on its range, \cref{rankreg} implies that high-rank polynomials behave  like independent random variables. See \cite{BFHHL13,hatami2014general} for further discussion of stronger equidistribution properties of high-rank polynomials.
%
%
%
Another way that high-rank polynomials behave like generic polynomials is that their restriction to subspaces preserves degree and high rank. We refer to \cite{BFHHL13} for a proof.
\begin{lemma}[Degree and rank preservation]\label{rankrestrict}
Suppose $f: \F^n \to \T$ is a polynomial of degree $d$ and rank
$\geq r$, where $r > q+1$. Let $A$ be a hyperplane in $\F^n$. Then,
$f\vert_A$ is a polynomial of degree $d$ and rank $\geq r-|\F|$, unless $d=1$
and $f$ is constant on $A$.
\end{lemma}
Bhowmick and Lovett~\cite{BL15largefield} showed that this can be improved in the case when $|\F|>d$. The next lemma combines this observation with the above lemma.
\begin{lemma}\label{rankrestrict}
Let $f\in \cP_{d}(\F^n\rightarrow \T)$ such that $\rank(f)\geq r$. Let $H$ be a hyperplane in $\F^n$. Then the restriction of $f$ to $H$ has rank at least $\max\{r-d-1, r-|\F|-1\}$.
\end{lemma}
The following is a surprising and very useful property of high-rank polynomials that was proved by Bhattacharyya, et. al.~\cite{BFHHL13}.
\begin{lemma}[Degree preservation, Lemma 2.13 of \cite{BFHHL13}]\label{lem:degreepreserve}
Let $d>0$ be given. There exists a nondecreasing function $r_{d,\F}:\N\rightarrow \N$ such that the following holds. Let $\cB$ be a rank $\geq r_{d,\F}$ polynomial factor defined by degree $\leq d$ (nonclassical) polynomials $P_1,...,P_m:\F^n \rightarrow \T$. Let $\Gamma:\T^n\rightarrow \T$.  Then
$$
\deg(\Gamma(Q_1(x),...,Q_m(x))) \leq \deg(\Gamma(P_1(x),...,P_m(x))),
 $$ 
 for every collection of polynomial $Q_1,...,Q_m:\F^n \rightarrow \T$, with $\deg(Q_i)\leq \deg(P_i)$ and $\depth(Q_i)\leq \depth(P_i)$.
\end{lemma}
We prove a lemma relating the $\nrank$ of a polynomial to its $\nrank$ over constant codimensional subspaces. 
\begin{lemma}\label{lem:subspace}
Let $f:\F^n\rightarrow \F$ be a degree $d$ polynomial and $V$ be an affine subspace of $\F^n$ of dimension $n-t$. Then,
$$
\nrank(f)\leq \nrank(f\vert_V)+ t.
$$
\end{lemma}
\begin{proof}
It suffices to prove that for a hyperplane $W$, $\nrank(f)\leq \nrank(f\vert_V)+1$. The lemma then simply follows by induction on $t$, the codimension of $V$. 

Suppose $W=\{x\in \F^n | \sum_{i=1}^n w_i x_i = a\}$, where $w\in \F^n$ and $a\in \F$. Applying an affine invertible projection, we can assume without loss of generality that $w= (1, 0, \ldots, 0)$ and $a=0$, and thus $W=\{ x\in \F^n\vert x_1=0\}$. Assume that $\nrank(f\vert_W)= r$, hence there exist nonconstant polynomials $G_1,...,G_r, H_1,...,H_r:W\to \F$ where $\deg(G_i)+\deg(H_i)\leq d$ and a degree $\leq d-1$ polynomial $Q:W\to \F$ such that 
$$
f\vert_W= \sum_{i=1}^r G_i H_i + Q.
$$
Now note that, 
$$
f(x_1,...,x_n)= f\vert_W (0,x_2,\ldots,x_n) + x_1 R(x_1,...,x_n),
$$
where $\deg(R)\leq d-1$. Thus 
$$
f= x_1 R+ \sum_{i=1}^r G_i H_i + Q,
$$
equivalently $\nrank(f)\leq r+1$.
\end{proof}
Another interpretation of the above lemma is that polynomials of high $\nrank$ are generic in a strong sense. We finally observe that  all the discussed notions of rank are subadditive.
\begin{claim}\label{claim:subadditive}
For every fixed vectors $a,b\in \F^n$, 
\begin{itemize}
\item[(i)] $\nrank(D_{a+b}f) \leq \nrank(D_af) + \nrank(D_b f)$.
\item[(ii)] $\crank(D_{a+b}f)\leq \crank(D_af) + \crank(D_b f)$.
\item[(iii)] $\rank(D_{a+b}f)\leq \rank(D_af) + \rank(D_b f)$.
\end{itemize}
\end{claim}
\begin{proof}
We compute $D_{a+b}f(x)$,
\begin{align*}
D_{a+b}f(x)&= f(x+a+b)- f(x)\\ 
&= f(x+a+b)-f(x+a) + f(x+a) - f(x)\\ 
&= D_bf(x+a) + D_af(x).
\end{align*}
The claim follows by observing that $\nrank(D_bf(x+a)) \leq \nrank(D_bf(x))$, $\crank(D_bf(x+a)) \leq \crank(D_bf(x))$, and $\rank(D_bf(x+a)) \leq \rank(D_bf(x))$, as the degrees of polynomials are preserved under affine shifts.
\end{proof}

\section{Structure of biased polynomials}
Throughout this section we will assume $\F=\F_q$ is a fixed prime field. By the discussion \cref{generalF}, the dependence on $|\F|$ can be removed from every step of our proof. \ignore{Further, by \cref{generalF} when $\F=\F_{q^m}$ is a fixed field, our proof can be extended, under the caveat that  the constants will depend on $|\F|$.}

We will need the following structure theorem for subsets of $\F^n$ with small doubling. For a set $A\subseteq \F^n$ and $k\geq 1$, denote the set $kA-kA:=\left\{a_1+\cdots+a_k-b_1-\cdots -b_k \vert a_1,...,a_k, b_1,...,b_k\in A\right\}$.
\begin{lemma}[Bogolyubov-Chang]\label{chang}
Let $A\subseteq \F^n$ such that $|A|=\mu |\F|^n$. Then, for some $k\leq \max(1,\lceil \frac{1}{2}(\log_{\frac{|\F|}{|\F|-1/2}}(\frac{2}{\mu})+2)\rceil)$, $kA-kA$ contains a subspace $V$ of $\F^n$ of co-dimension at most $\log_{\frac{|\F|-1/2}{|\F|-1}}(\frac{1}{2\mu})$.
\end{lemma}

The following lemma states that for a function $f:\F^n\rightarrow \F$ to be biased, there must be a positive set of directions $y$ for which $D_yf$ is somewhat biased.
\begin{lemma}\label{lem:biasedderivatives}
Suppose $f:\F^n\rightarrow \F$ is such that $\bias(f)=\delta$. Then there exists a set $A\subseteq \F^n$, with $|A|\geq \frac{\delta^2}{2}|\F|^n$ such that for every $y\in A$, $\bias(D_yf)\geq \frac{\delta^2}{2}$.
\end{lemma}
\begin{proof}
We compute the average bias of $D_yf$ for $y\in \F^n$ uniformly at random.
\begin{equation}
\E_{y\in \F^n} \left[ \bias(D_yf)\right] = \E_{y\in \F^n}\left[ \abs{\E_{x\in \F^n} e_\F(f(x+y)-f(x))}\right] 
\geq \abs{ \E_{z,x\in \F^n} \left[e_\F(f(z)) e_\F(-f(x))\right]}= \delta^2.
\end{equation}
Thus, since $\bias(f)\leq 1$, we get
\begin{equation}\label{eq:randomderiv}
\Pr_{y\in \F^n} \left[ \bias(D_yf)\geq \frac{\delta^2}{2}\right]\geq \frac{\delta^2}{2}.
\end{equation}
The lemma follows by choosing $A:=\{y\in \F^n \vert \bias(D_yf)\geq \frac{\delta^2}{2}\}\subseteq \F^n$.  
\end{proof}
We will use this lemma along with \cref{chang} and \cref{claim:subadditive} to show that for every biased function $f$ there exists a not too small subspace restricted to which all the derivatives of $f$ are biased.

\subsection{Structure of biased polynomials I, when $d<|\F|+4$}\label{section:generald}
In this section we prove that biased degree $d$ polynomials are strongly structured when $d<|\F|+4$. 

\restate{\cref{generald} [Biased degree $d$ polynomials I (when $d<|\F|+4$)]}{
Suppose $d>0$ and $\F=\F_{q}$ with $d< q+4$. Let $f:\F^n \rightarrow \F$ be a degree $d$ polynomial with $\bias(f)=\delta$. Then $\nrank(f)\leq c(\delta,d)$, namely there exists $c_{\ref{generald}}\leq c(\delta, d)$, nonconstant polynomials $G_1,...,G_c, H_1,...,H_c$ and a polynomial $Q$ such that the following hold.
\begin{itemize}
\item $f= \sum_{i=1}^c G_i H_i+Q$.
\item For every $i\in [c]$, $\deg(G_i)+\deg(H_i)\leq d$.
\item $\deg(Q)\leq d-1$.
\end{itemize}
Note that $c_{\ref{generald}}$ does not depend on $n$ or $|\F|$.
}

We will assume $\F=\F_p$ is a fixed prime field, and the constant $c=c(\delta,d,|\F|)$ we obtain will depend on $|\F|$. However by the discussion \cref{generalF}, it is straightforward to remove the dependence of $c(\delta,d,|\F|)$ on $|\F|$. \ignore{ Moreover, our proof can be extended to hold for non-prime field $\F=\F_{q^m}$ under the downside that  the constants will depend on $|\F|$. This dependence on $|\F|$ can be removed if $q>d$.}
\begin{proof}
By \cref{lem:biasedderivatives} there exists a set $A\subseteq \F^n$, with $|A|\geq \frac{\delta^2}{2}|\F|^n$ such that for every $y\in A$, 
$$
\bias(D_yf)\geq \frac{\delta^2}{2}.
$$
Thus by \cref{rankreg} for every $y\in A$,
$$
\crank(D_yf)\leq r= r_{\ref{rankreg}}(d,|\F|, \delta).
$$
Applying \cref{chang}, there is a subspace $V$ of co-dimension $t:=\log_{\frac{|\F|-1/2}{|\F|-1}}(\frac{1}{\delta^2})$ such that $V\subseteq kA-kA$, where $k=O(\log (\frac{1}{\delta^2}))$. By \cref{claim:subadditive} (ii), since $V\subseteq kA-kA$ we have that for every $y\in V$, 
$$
\crank(D_yf) \leq c_1\leq 2k r.
$$
By a simple averaging argument, there is an affine shift of $V$, $W:=V+h$ such that $\bias(f\vert_{W})\geq \delta$. Let us denote $\tf:= f\vert_W$. By \cref{lem:subspace}, it is sufficient to prove that $\nrank(\tf)\leq c_1(|\F|, \delta)$. Since $\bias(\tf)\geq \delta$,  \cref{rankreg} implies $\crank(\tf)\leq r_0=r_0(\delta,|\F|)$, moreover, there are $y_1,\ldots, y_{r_0}\in W$ and a $\Gamma:\F^{r_0}\rightarrow \F$ such that 
\begin{equation}\label{eq:rankinW}
\tf= \Gamma(D_{y_1}\tf, \ldots, D_{y_{r_0}}\tf).
\end{equation}

Note that for all $i\in [r_0]$,
\begin{equation}\label{eq:crank}
\crank_{d-1}(D_{y_i}\tf)\leq \crank(D_{y_i}f)\leq c_0
\end{equation}
This is due to the fact that an affine transformation can only decrease the degrees of polynomials and thus it can only decrease the $\crank$ of polynomials. 
\begin{remark}
We point out that the subscript $d-1$ in the LHS of \cref{eq:crank} is necessary, as can be seen by the following example. Suppose $d-1=4$, $m>0$ and $n=3m+4$.  Let $Q= x_{n-3}x_{n-2}x_{n-1}x_n+ \sum_{i=1}^m x_{3i-2}x_{3i-1}x_{3i}$. Now note that 
$$
\crank(Q)\leq 3,
$$
while
\begin{itemize}
\item $\crank(Q\vert_{x_n=0})= \crank(\sum_{i=1}^m x_{3i-2}x_{3i-1}x_{3i}) = \omega_n(1),$ since $\norm{e_\F(\sum_{i=1}^m x_{3i-2}x_{3i-1}x_{3i})}_{U^3} = o(1)$.
\item $\crank_4(Q\vert_{x_n=0})=1$, since $\deg(Q\vert_{x_n=0})<4$.
\end{itemize}
\end{remark}
By \cref{eq:crank} there exist degree $\leq d-2$ polynomials 
$\set{\iG_1,\ldots, \iG_{c_0}}_{i=1}^{r_0}$
and a function $\Lambda:\F^{r_0c_0}\rightarrow \F$ such that
\begin{equation}\label{eq:d-2factor}
\tf = \Lambda\left( (\iG_1,\ldots, \iG_{c_0})_{i=1}^{r_0}\right).
\end{equation}
We would like to regularize this collection of polynomials, however we would like to avoid any appearance of nonclassical polynomials. The following observation allows us to do exactly that as long as $d< |\F|+4$.
\begin{claim}[Nonclassical regularity lemma over large characteristic]\label{generaldfactorreg}
Let $r: \N \to \N$ be a non-decreasing function. And $d$ be such that $d<|\F|+4$. Then, there is a  function
$C_{\ref{generaldfactorreg}}^{\F,r}: \N \to \N$  such
that the following holds. Suppose $\cB$ is a factor defined by classical polynomials $P_1,\dots, P_C : \F^n \to \T$  of degree at most $d-2$.
Then, there is  an $r$-regular factor $\cB'$ consisting only of classical polynomials
$Q_1, \dots, Q_{C'}: \F^n \to \T$ of degree $\leq d-2$ such that $\cB'
\succeq_{sem} \cB$ and  $C' \leq  C_{\ref{generaldfactorreg}}^{(\F,r)}(C)$.
\end{claim}
\begin{remark}
Note that the above claim does not hold for general degrees, as we require the obtained factor be high-rank as defined in \cref{def:rankpoly}, which is complexity against nonclassical polynomials. To see this, we observe that in the case of quartic polynomials, the single polynomial $\{S_4\}$ cannot be refined to a high-rank polynomial factor defined by $O(1)$ {\em classical} polynomials. However, it can be refined to a high-rank nonclassical factor by \cref{factorreg}. This is the barrier to extending our results to sextic and higher-degree polynomials. Starting with a biased sextic polynomial, dealing with non-classical polynomials seems to be unavoidable.
\end{remark}
We postpone the proof of \cref{generaldfactorreg} and show how it can be used to conclude \cref{main}. Fix $r_1:\N\to \N$ a nondecreasing function as in \cref{lem:degreepreserve} for degree $d-2$. Let $\cB$ be the factor defined by degree $\leq d-2$ classical polynomials $\{\iG_1,\ldots, \iG_{c_0}\}_{i=1}^{r_0}$.  Applying \cref{generaldfactorreg} to $\cB$ with regularity parameter $r_1$, we obtain a refinement $\cB' \succeq_{sem} \cB$, where $\cB'$ is defined by $c_2:= C_{\ref{generaldfactorreg}}^{(\F,r_1)}(c_0r_0)$ classical degree $\leq d-2$ polynomials $R_1,\ldots, R_{c_2}:\F^n\rightarrow \F$. Namely, there exists a function $\cK:\F^{c_2}\to \F$, such that
$$
\tf= \cK(R_1,\ldots, R_{c_2}).
$$
Applying an affine transformation, assume without loss of generality that $W=\set{x\in \F^n\vert x_1=x_2=\cdots=x_{t}=0}$. Moreover, we may assume that $n-t>c_2$, since otherwise, $\tf$ has at most $d(n-t)^d= O(c_2^d)$ monomials, making the theorem statement trivial. For every $i\in [c_2]$, let $d_i:= \deg(R_i)$, $s_i:=\sum_{j=1}^i d_i$, and define $R'_i:= x_{s_{i-1}+1}\cdots x_{s_i}$. We have that $\deg(R'_i)=\deg(R_i)$ and thus by \cref{lem:degreepreserve}, 
$$
\deg(\cK(R'_1,\ldots, R'_{c_2})) \leq \deg(\cK(R_1,\ldots, R_{c_2}))= \deg(\tf) = d.
$$
Note that $\cK:\F^{c_2}\to \F$ is a polynomial, and $R'_1,...,R'_{c_2}$ are monomials on disjoint variables, thus plugging in $R'_i$s into $\cK$'s variables, no cancelations can occur. In particular, 
$$
\cK(y_1,\ldots,y_{c_2})= \sum_{s\in \{0,...,q-1\}^{c_2}, \sum_{i} s_i d_i\leq d} \alpha_s \prod_{i\in S} y_i^{s_i},
$$ 
where $\alpha_S\in \F$ are coefficients of $\cK$. Hence,
\begin{equation}\label{eq:strongform}
\tf=\cK(R_1,\ldots, R_{c_2})= \sum_{s\in \{0,...,q-1\}^{c_2}, \sum_{i} s_i d_i\leq d} \alpha_s \prod_{i\in S} R_i^{s_i}.
\end{equation}
Namely, $\nrank(\tf)\leq dc_2^d$, and by \cref{lem:subspace} we deduce $\nrank(f)\leq dc_2^d+t$ as desired.
\end{proof}

\begin{proofof}{of \cref{generaldfactorreg}}
We observe that the iterative proof of \cref{factorreg} can be modified to include only classical polynomials. \cref{factorreg} is proved by a transfinite induction on the vector of number of (possibly nonclassical) polynomials of each degree and depth defining the polynomial factor. One then argues that a polynomial factor that is not of the desired rank, can always be refined to a polynomial factor where some polynomial is replaced by a collection of polynomials that are of either lower degree, or same degree with lower depth. 

We observe that if we start with a polynomial factor defined by degree $\leq d-2$ classical polynomials, the only nonclassical polynomials that may arise are of degree $d-3\leq |\F|$ and thus of depth $1$, this is due to the fact that any nonclassical polynomial of depth $\geq 2$ has degree $\geq 2|\F|-1$. Now we use a known fact that polynomials of degree $|\F|$ that are not classical are unncessary in higher order Fourier analysis. More precisely in~\cref{inverse}, for the case of degree $|\F|$ polynomials, one can assume that the polynomial $P:\F^n \to \T$ in the statement of the theorem is a \emph{classical} polynomial of degree at most $\leq |\F|$.  More generally \cite{MR3284051} showed a similar fact for higher depths.
\begin{theorem}[Unnecessary depths~\cite{MR3284051}]\label{lem:unnecessarydepths}
Let $k\geq 1$, and $q$ the characteristic of $\F$. Every nonclassical polynomial $f:\F^n\rightarrow \T$ of degree $1+k(q-1)$ and depth $k$, can be expressed as a function of three degree $\leq 1+k(q-1)$ polynomials of depth $\leq k-1$.
\end{theorem}
By the above discussion we may assume that in our application of \cref{factorreg}, $\cB'$ is defined via only classical polynomials.
\end{proofof}


\subsection{Structure of biased polynomials II, when $d<|\F|+4$}\label{section:generald2}
In this section we prove that a biased degree $d$ polynomial is constant on a large subspace.
 
\restate{\cref{generald2} [Biased degree $d$ polynomials II (when $d<|\F|+4$)]}{ 
Suppose $d>0$ and $\F=\F_q$ with $d<q+4$. Let $f:\F^n \rightarrow \F$ be a degree $d$ polynomial with $\bias(f)=\delta$. There exists an affine subspace $V$ of dimension $\Omega_{d,\delta}(n^{1/\lfloor \frac{d-2}{2}\rfloor})$ such that $f\vert_V$ is a constant.
}

In the case of $d=5$ we have $5<2+4\leq |\F|+4$ and $\lfloor(d-2)/2\rfloor=1$, hence we obtain a subspace of dimension $\Omega_{\delta}(n)$ as desired in \cref{main2}.

We will need the following result of Cohen and Tal~\cite{CT15} on the structure of low degree polynomials. 
\begin{theorem}[\cite{CT15}, Theorem 3.5]\label{thm:avishay}
Let $q$ be a prime power. Let $f_1,\ldots, f_\ell:\F_q^n\rightarrow \F_q$ be polynomials of degree $d_1,\ldots, d_\ell$ respectively. Let $k$ be the least integer such that
$$
n\leq k+ \sum_{j=0}^{\ell} (d_i+1)\sum_{j=0}^{d_i-1}(d_i-j)\cdot {k+j-1 \choose j}.
$$
Then, for every $u_0\in \F_q^n$ there exists a subspace $U\subseteq \F_q^n$ of dimension $k$, such that for all $i\in [\ell]$, $f_i$ restricted to $u_0+U$ is a constant function. 

In particular, if $d_1,...,d_\ell\leq d$, then the above holds for $k=\Omega((n/\ell)^{\frac{1}{d-1}})$.
\end{theorem}
\ignore{In \cite{CT15}, the above theorem was used to prove that biased degree three and four polynomials vanish over a large subspace. 
\begin{lemma}[\cite{CT15}]\label{lem:avishay}
Let $f:\F^n\to \F$ be a degree $d\leq 4$ polynomial with $\bias(f)=\delta$. If $d=3$ then there is a subspace $V$ of dimension $n-O(\log(1/\delta)^2)$  such that $f\vert_V$ is a constant. If $d=4$, then there is a subspace $V$ of dimension $\frac{n}{poly(1/\delta)}$ such that $f\vert_V$ is a constant.
\end{lemma}}

\begin{proofof}{of \cref{generald2}}
\ignore{We induct on $d$. The base case of $d\leq 4$ follows from \cref{thm:avishay,lem:avishay}. }Following the proof of \cref{generald}, there exists  an affine subspace $W$ of dimension $n-t$ for $t=poly(\log(\frac{1}{\delta^2}))$, for which \cref{eq:strongform} holds. By \cref{rankreg}, choosing a proper regularity parameter in the application of \cref{generaldfactorreg}, we can further assume that the factor defined by $R_1,...,R_{c_2}$ is $\frac{\delta}{2}q^{-c_2}$-unbiased in the sense of \cref{dfn:uniformfactor}. We may rewrite \cref{eq:strongform} in the form
$$
f\vert_W =  \sum_{i=1}^{C} \alpha_i G_i H_i + M,
$$
where $C\leq c_2^d$, $\alpha_i$ are field elements, $M$ is a degree $\leq d-2$ polynomial, $G_i$s and $H_i$s are nonconstant degree $\leq d-2$ polynomials satisfying $\deg(G_i)+\deg(H_i)\leq d$. Moreover, every $G_i$ and $H_i$ is product of a subset of $\{R_1,...,R_{c_2}\}$. We crucially observe that $M$ can be taken to be of the form 
$$
M = \sigma_0+\sum_{i=1}^{c_2} \sigma_i R_i,
$$
where $\sigma_i$ are field elements, such that $\sigma_i\neq 0$ implies that $R_i$ does not appear in $\sum_{i=1}^C \alpha_i G_i H_i$.  
\begin{claim}
Let $f$, $W$, $R_1,...R_{c_2}$ and $M$ be as above. Then $M$ is a constant.
\end{claim}
\begin{proof}
Assume for contradiction that $M$ is nonconstant. By the above discussion, letting $$S:=\{j\in [c_2]: R_j \text{ appears in } \sum_i \alpha_i G_i H_i\},$$ we have
$$
f\vert_W = \Lambda(R_j)_{j\in S} + \sum_{i\in [c_2]\backslash S} \sigma_j R_j,
$$
for some function $\Lambda:\F^{|S|}\to \F$. Writing the Fourier expansion of $e_\F(\Lambda)$, we have
$$
e_\F(f\vert_W)= \sum_{\gamma\in \F^{|\S|}} \widehat{\Lambda}(\gamma) e_\F(\sum_{j\in S}\gamma_j R_j+M).
$$
Note that $W$ was chosen such that $\bias(f\vert_W)\geq \delta$. Thus,
\begin{align*}
\bias(f\vert_W)&= \abs{\Ex_{x\in \F^n} e_\F(\Lambda(R_j)_{j\in S} + M)} \\ &=
\abs{\E_{x\in \F^n} \sum_{\gamma\in \F^{|S|}} \widehat{\Lambda}(\gamma) e_\F(M+\sum_{j\in S} \gamma_j R_j)}\\ &\leq
\sum_{\gamma\in \F^{|S|}} \widehat{\Lambda}(\gamma) \bias(M+\sum_{j\in S} \gamma_j R_j)\\ &\leq 
q^{c_2} \cdot \frac{\delta}{2}q^{-c_2}<\delta,
\end{align*}
contradicting $\bias(f\vert_W)=\delta$, where the last inequality uses the fact that the factor defined by $R_1,...,R_{c_2}$ is $\frac{\delta}{2}q^{-c_2}$-unbiased.
\end{proof}
By the above claim $M$ is a constant, and thus
$$
f\vert_W= \sigma_0+ \sum_{i=1}^C \alpha_i G_iH_i.
$$
Recall that $\deg(G_i)+\deg(H_i)\leq d$, hence for every $i$,  $\min\{\deg(G_i),\deg(H_i)\}\leq \lfloor\frac{d}{2}\rfloor$. Thus by \cref{thm:avishay}, there is an $\Omega_{C}((n-t)^{1/\lfloor \frac{d-2}{2} \rfloor})=\Omega_{\delta, \F,d}(n^{1/\lfloor \frac{d-2}{2} \rfloor})$ dimensional affine subspace $W'$ such that $f\vert_{W'}$ is constant. 
\end{proofof}

\section{Algorithmic Aspects}\label{section:algorithmic}
In this section we show that the strong structures implied by \cref{main} and \cref{generald} can be found by a deterministic algorithm that runs in time polynomial in $n$.

\restate{\cref{algorithmic}}{
Suppose $\delta>0$, $d>0$ are given, and let $\F=\F_q$ be a prime field satisfying $d<q+4$. There is a deterministic algorithm that runs in time $O(n^{O(d)})$ and given as input a degree $d$ polynomial $f:\F^n\to \F$ satisfying $\bias(f)=\delta$, outputs a number $c\leq c(\delta, |\F|,d)$, a collection of degree $\leq d-1$ polynomials $G_1,...,G_c, H_1,...,H_c:\F^n\to \F$ and a polynomial $Q:\F^n\to \F$, such that
\begin{itemize}
\item $f= \sum_{i=1}^c G_iH_i +Q$.
\item For every $i\in [c]$, $\deg(G_i)+\deg(H_i)\leq d$.
\item $\deg(Q)\leq d-1$.
\end{itemize}}
\begin{proof}
We will use the following result of Bhattacharyya, et. al.~\cite{BHT15} who proved several algorithmic regularity lemmas for polynomials.
\begin{theorem}[\cite{BHT15}, Theorem 1.6]\label{algorithmicdecomposition}
For every finite field $\cF$ of fixed prime order, positive integers $d,k$, every vector of positive integers $\Delta=(\Delta_1,...,\Delta_k)$, and every function $\Gamma:\F^k\rightarrow \F$, there is a deterministic algorithm that takes as input a polynomial $f:\F^n\rightarrow \F$ of degree $d$, runs in time polynomial in $n$, and outputs polynomials $Q_1,...,Q_k$ of degrees respectively at most $\Delta_1,..., \Delta_k$ such that
$$
f= \Gamma(Q_1,...,Q_k),
$$
if such a decomposition exists, while otherwise accurately returning \em{\large NO}.
\end{theorem}
By \cref{generald}, we know that there is $c\leq C(\delta, |\F|, d)$ such that there exist a collection of nonconstant polynomials $G_1,...,G_c, H_1,...,H_c:\F^n\to \F$, and a polynomial $Q:\F^n\to \F$, such that
\begin{equation}\label{alg:strongstructure}
f= \sum_{i=1}^c G_iH_i +Q,
\end{equation}
for every $i\in [c]$, $\deg(G_i)+\deg(H_i)\leq d$, and $\deg(Q)\leq d-1$.
The algorithm is now straight-forward.

\begin{itemize}
\item[{\bf 1}] Iterate through all choices for $c\leq C(\delta, |\F|,d)$. This is our guess for the number of terms in the summation in \cref{alg:strongstructure}.
\begin{itemize}
\item[{\bf 1.1}] Iterate through all choices of $d_1,\ldots, d_c, d'_1,\ldots, d'_c\leq d-1$ and $d''\leq d-1$ such that $d_i+d'_i\leq d$. These are our guesses for degree sequences for $G_1,...,G_c,H_1,...,H_c$ and $Q$. Note that this step does not depend on $n$. 
\begin{itemize}
\item[{\bf 1.1.1}] Define $\Gamma:\F^{2c+1} \rightarrow \F$ as 
$$
\Gamma(x_1,\ldots, x_c, y_1, \ldots, y_c, z):= \sum_{i=1}^c x_iy_i + z.
$$
\item[{\bf 1.1.2}] Run \cref{algorithmicdecomposition} on the polynomial $f$, with $\Delta=(d_1,\ldots, d_c, d'_1,\ldots, d'_c, d'')$ and $\Gamma$ as inputs.
\begin{itemize}
\item[{\bf 1.1.2.a}] If the algorithm outputs { \large NO}, then continue.
\item[{\bf 1.1.2.b}] If the algorithm outputs a collection of polynomials satisfying the decomposition, halt and output the desired decomposition.
\end{itemize}
\end{itemize}
\end{itemize}
\end{itemize}
By \cref{generald} and \cref{algorithmicdecomposition} the above algorithm will always halt with a decomposition of desired form. The number of possible choices in {\bf 1} and {\bf 1.1} do not depend on $n$, and step {\bf 1.1.2} runs in polynomial time in $n$, as a result making the algorithm polynomial time in $n$.
\end{proof}

\section{Conclusions}\label{section:conclusions}
Green and Tao~\cite{MR2592422} and Kaufman and Lovett~\cite{KL08} proved that every degree $d$ polynomial $f$ with $\bias(f)=\delta$ can be written in the form 
\begin{equation}\label{conclusions}
f=\Gamma(P_1,..., P_c),
\end{equation}
for $c\leq c(\delta, d,\F)$ and degree $\leq d-1$ polynomials $P_1,...,P_c$. However, nothing is known on the structure of the function $\Gamma$ in \cref{conclusions}. In this work we showed that in the case of degree five polynomials we can say much more about the structure of $f$. More generally for degree $d$ polynomials when $d<|\F|+4$, we can write 
$$
f=\sum_{i=1}^{C} G_i H_i+Q,
$$ 
for nontrivial polynomials $G_i, H_i$ satisfying $\deg(G_i)+\deg(H_i)\leq d$, and $\deg(Q)\leq d-1$. It is a fascinating question whether similar structure theorems hold in the case of $d\geq |\F|+4$, more specifically we suspect that answering this question for degree $6$ polynomials and $\F=\F_2$ will suffice resolve the question for all degrees and characteristics.

\begin{openproblem}
Can every biased degree six polynomial $f:\F_2^n\to \F_2$ be written in the form 
$$
f=\sum_{i=1}^{C} G_i H_i+Q,
$$ 
for $C\leq C(\bias(f))$, nontrivial polynomials $G_i, H_i$ satisfying $\deg(G_i)+\deg(H_i)\leq 6$, and $\deg(Q)\leq 5$?
\end{openproblem}
A somewhat weaker question that also remains open is whether we can bound the degree of $\Gamma$ in \cref{conclusions} in terms of $d$ only.
\begin{openproblem}
Suppose that $\F=\F_q$ for a prime $q$. Can every degree $d$ polynomial $f:\F^n\to \F$ be written in the form
$$
f=\Gamma(P_1,...,P_{C_1}),
$$
where $C\leq C(\bias(f), \F,d)$, $P_1,...,P_C$ are degree $\leq d-1$ polynomials, and 
$
\deg(\Gamma)\leq O_d(1)?
$
\end{openproblem}

Finally, we note that the constants obtained in \cref{main,main2,generald,generald2}, unlike \cref{HS:cubic} and \cref{HS:quartic}, have very bad dependence on $\delta$ and $d$. In particular, in the case of degree five polynomials, an interesting problem that remains unaddressed is to find out what the optimum constant achievable in \cref{main} is.

\bibliographystyle{amsalpha}
\bibliography{bibs}
\end{document}